\documentclass[12pt, reqno]{amsart}
\usepackage{amsmath, amsthm, amscd, amsfonts, amssymb, graphicx, color}
\usepackage[bookmarksnumbered, colorlinks, plainpages]{hyperref}

\newtheorem{theorem}{Theorem}[section]

\newtheorem{proposition}[theorem]{Proposition}
\newtheorem{corollary}[theorem]{Corollary}
\theoremstyle{definition}
\newtheorem{definition}[theorem]{Definition}
\newtheorem{example}[theorem]{Example}

\theoremstyle{remark}

\numberwithin{equation}{section}

\begin{document}
\setcounter{page}{1}

\title[Some Hermite-Hadamard type inequalities]{Some Hermite-Hadamard type inequalities for the product of two operator preinvex functions }
\author[A.Ghazanfari \and  A.Barani]{A. G. Ghazanfari$^{1,*}$, A. Barani$^{1}$}
\address{$^{1}$ Department of
Mathematics, Lorestan University, P.O.Box 465, Khoramabad, Iran.}
\email{\textcolor[rgb]{0.00,0.00,0.84}{ghazanfari.a@lu.ac.ir; barani.a@lu.ac.ir}}


\subjclass[2010]{Primary 47A63; Secondary 26D07, 26D15.}

\keywords{Hermite-Hadamard inequality, invex sets, operator preinvex functions.}

\date{Received: xxxxxx; Revised: yyyyyy; Accepted: zzzzzz.
\newline \indent $^{*}$ Corresponding author}

\begin{abstract}
In this paper we introduce operator preinvex
functions and establish a Hermite-Hadamard type inequality for such functions. We give an estimate of the right hand
side of a Hermite-Hadamard type inequality in which some operator preinvex
functions of selfadjoint operators in Hilbert spaces are involved. Also some Hermite-Hadamard type inequalities
for the product of two operator preinvex functions are given.
\end{abstract}\maketitle
\section{introduction}

The following inequality holds for any convex function $f$ defined on $\mathbb{R}$ and $a,b \in \mathbb{R}$, with $a<b$
\begin{equation}\label{1.1}
f\left(\frac{a+b}{2}\right)\leq \frac{1}{b-a}\int_a^b f(x)dx\leq \frac{f(a)+f(b)}{2}
\end{equation}
Both inequalities hold in the reversed direction if $f$ is concave. We note that
Hermite-Hadamard's inequality may be regarded as a refinement of the concept of convexity and it follows easily from
Jensen's inequality.
The classical Hermite-Hadamard inequality provides estimates of the mean value
of a continuous convex function $f : [a, b] \rightarrow \mathbb{R}$.
The Hermite-Hadamard inequality has several applications in nonlinear analysis and
the geometry of Banach spaces, see \cite{kik, con}.

In recent years several extensions and generalizations have been considered for
classical convexity. We would like to refer the reader to \cite{bar, dra2, wu} and references
therein for more information. A number of papers have been written on this inequality providing
some inequalities analogous to Hadamard's inequality
given in (\ref{1.1}) involving two convex functions, see \cite{pach, bak, tun}.
Pachpatte in \cite{pach} has proved the following theorem for the product of two convex functions.
\begin{theorem}\label{t1}
Let $f$ and $g$ be real-valued, nonnegative and convex functions on $[a, b]$. Then
\begin{equation*}
\frac{1}{b-a}\int_{a}^{b}f(x)g(x)dx\leq\frac{1}{3}M(a, b)+\frac{1}{6}N(a, b),
\end{equation*}
\begin{equation*}
2f\left(\frac{a+b}{2}\right)g\left(\frac{a+b}{2}\right)\leq\frac{1}{b-a}\int_{a}^{b}f(x)g(x)dx+\frac{1}{6}M(a, b)+\frac{1}{3}N(a, b),
\end{equation*}
where $M(a, b)=f(a)g(a)+f(b)g(b), N(a, b)=f(a)g(b)+f(b)g(a)$.
\end{theorem}

A significant generalization of convex functions is that of invex
functions introduced by Hanson in \cite{han}.
In this paper we introduce operator preinvex functions and give an operator version
of the Hermite-Hadamard inequality for such functions.

First, we review the operator order in $B(H)$ and the continuous functional calculus for a bounded selfadjoint operator.
For selfadjoint operators $A, B \in B(H)$ we write $A\leq B ( \text{or}~ B\geq A)$ if $\langle Ax,x\rangle\leq\langle Bx,x\rangle$ for every vector $x\in H$,
we call it the operator order.

Now, let $A$ be a bounded selfadjoint linear operator on a complex Hilbert space $(H; \langle .,.\rangle)$
and $C(Sp(A))$ the $C^*$-algebra of all continuous complex-valued functions on the spectrum of $A$.
The Gelfand map establishes a $*$-isometrically isomorphism $\Phi$ between
$C(Sp(A))$ and the $C^*$-algebra $C^*(A)$ generated by $A$  and the identity operator $1_H$ on $H$ as
follows (see for instance \cite[p.3]{fur}):
For $f, g\in C (Sp (A))$ and $\alpha,\beta\in\mathbb{C}$
\begin{align*}
&(i)~\Phi(\alpha f+\beta g)=\alpha\Phi(f)+\beta\Phi(g);\\
&(ii)~\Phi(fg)=\Phi(f)\Phi(g)~\text{and}~\Phi(f^*)=\Phi(f)^*;\\
&(iii)~\|\Phi(f)\|=\|f\|:=\sup_{t\in Sp(A)} |f(t)|;\\
&(iv)~\Phi(f_0)=1 ~\text{and}~ \Phi(f_1)=A,~\text{ where } f_0(t)=1~ \text{and } f_1(t)=t, \text{for}~\\
&t\in Sp(A).
\end{align*}
If $f$ is a continuous complex-valued functions on $Sp(A)$, the element $\Phi(f)$ of $C^*(A)$ is denoted by
$f(A)$, and we call it the continuous functional calculus for a bounded selfadjoint operator $A$.

If $A$ is a bounded selfadjoint operator and $f$ is a real-valued continuous function on $Sp (A)$,
then $f (t) \geq 0$ for any $t \in Sp (A)$ implies that $f (A) \geq 0$, i.e., $f (A)$ is a positive
operator on $H$. Moreover, if both $f$ and $g$ are real-valued functions on $Sp (A)$ such that
$f(t) \leq g (t)$ for any $t\in sp(A)$, then $ f(A)\leq g(A)$ in the operator order in $B(H)$.

A real valued continuous function $f$ on an interval $I$ is said to be operator convex
(operator concave) if
\begin{equation*}
f((1 -\lambda)A +\lambda B)\leq (\geq) (1-\lambda) f (A) + \lambda f (B)
\end{equation*}
in the operator order in $B(H)$, for all $ \lambda\in [0, 1] $ and for every bounded self-adjoint operators $A$ and $B$
in $B(H)$ whose spectra are contained in $I$.

For some fundamental results on operator convex (operator concave) and operator monotone functions, see \cite{fur, bha} and
the references therein.

Dragomir in \cite{dra1} has proved a Hermite-Hadamard type inequality for operator convex functions:

\begin{theorem}\label{t2}
Let $f : I \rightarrow \mathbb{R}$ be an operator convex function on the interval $I$. Then
for any selfadjoint operators $A$ and $B$ with spectra in $I$ we have the inequality
\begin{multline*}
\left(f\left(\frac{A+B}{2}\right)\leq\right)\frac{1}{2}\left[f\left(\frac{3A+B}{4}\right)+f\left(\frac{A+3B}{4}\right)\right]\\
\leq\int_0^1 f((1-t)A+tB))dt\\
\leq\frac{1}{2}\left[f\left(\frac{A+B}{2}\right)+\frac{f(A)+f(B)}{2}\right]\left(\leq\frac{f(A)+f(B)}{2}\right).
\end{multline*}
\end{theorem}
Moslehian in \cite{mos} generalized the above theorem \ref{t2} as follows:

\begin{theorem}\label{t3}
If $A,B$ are self-adjoint operators on a Hilbert space $H$ with spectra
in an interval $J$, $f$ is an operator convex function on $J$ and $k, p$ are positive integers,
then
\begin{multline*}
f\left(\frac{A+B}{2}\right)\leq\frac{1}{k^p}\sum_{i=0}^{k^p-1}f\left(\frac{2i+1}{2k^p}A+\left(1-\frac{2i+1}{2k^p}\right)B\right)\\
\leq\int_0^1 f((1-t)A+tB))dt\\
\leq\frac{1}{2k^p}\sum_{i=0}^{k^p-1}\left[f\left(\frac{i+1}{k^p}A+\left(1-\frac{i+1}{k^p}\right)B\right)
+f\left(\frac{i}{k^p}A+\left(1-\frac{i}{k^p}\right)B\right)\right]\\
\leq\frac{f(A)+f(B)}{2}.
\end{multline*}
\end{theorem}
Motivated by the above results we investigate in this paper the operator version
of the Hermite-Hadamard inequality for operator preinvex functions.
We show that Theorem \ref{t3} holds for operator preinvex functions and establish an estimate of the right
hand side of a Hermite-Hadamard type inequality in which some operator preinvex
functions of selfadjoint operators in Hilbert spaces are involved. We also give some Hermite-Hadamard type inequalities for the product of two
operator preinvex functions.

\section{operator preinvex functions}

\begin{definition}\label{d1}
Let $X$ be a real vector space, a set $S\subseteq X$ is said to be invex with respect
to the map $\eta:S\times S\rightarrow X$, if for every $x,y\in S$
and $t\in[0,1]$,
\begin{equation}\label {2.1}
y+t\eta(x,y)\in S.
\end{equation}
\end{definition}
It is obvious that every convex set is invex with respect
to the map $\eta(x,y)=x-y$, but there exist invex sets which are not convex (see \cite{ant}).

Let $S\subseteq X$ be an invex set with respect to $\eta:S\times S\rightarrow X$.
For every $x,y\in S$ the $\eta-$path $P_{xv}$ joining the points $x$ and $v : =x+\eta(y,x)$  is defined as follows
\[
P_{xv} : = \{z: z = x+t\eta(y,x) : t \in [ 0, 1 ] \}.
\]
The mapping $\eta$ is said to be satisfies the condition  $C$ if for every  $x,y\in S$ and $t \in[0,1]$,

\begin{align*}
(C)\quad\quad \eta(y,y+t\eta(x,y))&=-t\eta(x,y),\\
\eta(x,y+t\eta(x,y))&=(1-t)\eta(x,y).
\end{align*}

Note that for every $x,y\in S$ and every $t_{1},t_{2}\in [ 0, 1 ]$ from condition $C$ we have
\begin{equation}\label{2.2}
\eta(y+t_{2}\eta(x,y),y+t_{1}\eta(x,y))=(t_{2}-t_{1})\eta(x,y),
\end{equation}
see \cite{moh, yan} for details.

Let $\mathcal{A}$ be a $C^*$-algebra, denote by $\mathcal{A}_{sa}$ the set of all self adjoint elements in $\mathcal{A}$.
\begin{definition}\label{d2}
Let $I$ be an interval in $\mathbb{R}$ and $S\subseteq B(H)_{sa}$ be an invex set with respect to $\eta:S\times S\rightarrow B(H)_{sa}$. A continuous function $f:I\rightarrow \mathbb{R}$ is said to be operator preinvex on $I$ with respect to $\eta$ for operators in $S$ if
\begin{equation}\label {2.3}
f(B+t\eta(A,B))\leq (1-t)f(B)+tf(A).
\end{equation}
in the operator order in $B(H)$, for all $t\in[0,1]$ and for every $A,B\in S$ whose spectra are contained in $I$.
\end{definition}
Every operator convex function is an operator preinvex with respect
to the map $\eta(A,B)=A-B$ but the converse does not holds (see the following example).

Now, we give an example of some operator preinvex functions and invex sets with respect to the maps $\eta$ which satisfy the conditions (C).
\begin{example}\label{e2}\rm{
\begin{enumerate}
\item[(a)] Suppose that $1_H$ is the identity operator on a Hilbert space $H$, and
\begin{align*}
T:&=\{A\in B(H)_{sa}: A\leq -1\times1_H\}\\
U:&=\{A\in B(H)_{sa}: 1_H\leq A\}\\
S:&=T\cup U\subseteq B(H)_{sa}.
\end{align*}
 Suppose that the function $\eta_1:S\times S\rightarrow B(H)_{sa}$ is defined by
 \[
\eta_1(A,B) =
\begin{cases}
A-B & \text{ \( A,B\in U \),}\\
A-B & \text{ \( A,B\in T \),}\\
1_H-B & \text{ \( A\in T, B\in U\),}\\
-1_H-B & \text{ \( A\in U, B\in T\)}.
\end{cases}
\]
Clearly $\eta_1$ satisfies condition $C$ and $S$ is an invex set with respect to $\eta_1$.
We show that the real function $f(t)=t^2$ is operator preinvex with respect to $\eta_1$ on every interval $I\subseteq\mathbb{R}$, for operators in $S$. Since $f$ is an operator convex function on $I$,
for the cases which $ \eta_1(A,B)=A-B$ the inequality (\ref{2.3}) holds. Let $ \eta_1(A,B)=1_H-B$, in this case we have $1_H\leq-A\leq A^2$ and
\begin{align*}
(B+t\eta_1(A,B))^2=(B+t(1_H-B))^2=((1-t)B+t1_H)^2\\
\leq (1-t)B^2+t1_H\leq(1-t)B^2+tA^2.
\end{align*}
Similarly, for the case $\eta_1=-1_H-B$ we have
\begin{align*}
(B+t\eta_1(A,B))^2=(B+t(-1_H-B))^2=((1-t)(-B)+t1_H)^2\\
\leq (1-t)B^2+t1_H\leq(1-t)B^2+tA^2,
\end{align*}
therefore, the inequality (\ref{2.3}) holds.

 But the real function $g(t)=a+bt,\quad a,b\in \mathbb{R}$ is not operator
preinvex with respect to $\eta_1$ on $S$.
\item[(b)] Suppose that  $V:=(-2\times1_H,0),~ W:=(0,2\times1_H),~ S:=V\cup W\subseteq B(H)_{sa}$ and the
function $\eta_2:S\times S\rightarrow B(H)_{sa}$ is defined by
 \[
\eta_2(A,B) =
\begin{cases}
A-B & \text{ \( A,B\in V \text{ or }A,B\in W\),}\\
0   & \text{ otherwise }.
\end{cases}
\]
Clearly $\eta_{2}$ satisfies condition $C$ and $S$ is an invex set with respect to $\eta_2$.
The constant functions $f(t)=a,~a\in \mathbb{R}$ is only operator preinvex functions with respect to $\eta_2$ for operators in  $S$.
Because for $\eta_2=0$,
\begin{align*}
f(B+t\eta_2(A,B))=f(B)\leq (1-t)f(B)+tf(A),
\end{align*}
implies that $f(A)-f(B)\geq 0$. interchanging $A$ ,$B$ we get $f(B)-f(A)\geq 0$.
\item[(c)] The function $f(t) = -|t|$ is not a convex function, but it is a operator preinvex function with
respect to $\eta_3$, where
\[
\eta_3(A,B) =
\begin{cases}
A-B & \text{ \( A,B\geq 0 \text{ or }A,B\leq 0\),}\\
B-A   & \text{ otherwise \(A\leq 0\leq B\)\text{ or }\(B\leq 0\leq A\)}.
\end{cases}
\]

\end{enumerate}
 }
\end{example}

Let $X$ be a vector space, $ x, y \in X$, $ x\neq y$. Define the segment
\[[x, y] :=(1- t)x + ty; t \in [0, 1].\]
We consider the function $f : [x, y]\rightarrow \mathbb{R}$ and the associated function
\begin{align*}
&g(x, y) : [0, 1] \rightarrow \mathbb{R},\\
&g(x, y)(t) := f((1 - t)x + ty), t \in [0, 1].
\end{align*}
Note that f is convex on $[x, y]$ if and only if $g(x, y)$ is convex on $[0, 1]$.
For any convex function defined on a segment $[x, y] \in X$, we have the Hermite-
Hadamard integral inequality

\begin{equation}\label{2.4}
f\left(\frac{x+y}{2}\right)\leq \int_0^1 f((1-t)x+ty)dt\leq \frac{f(x)+f(y)}{2},
\end{equation}

which can be derived from the classical Hermite-Hadamard inequality (1.1) for the
convex function $g(x, y) : [0, 1] \rightarrow \mathbb{R}$.

\begin{proposition}\label{p1}
Let $I$ be an interval in $\mathbb{R}$, $S\subseteq B(H)_{sa}$ an invex set with respect to $\eta:S\times S\rightarrow B(H)_{sa}$ and $\eta$ satisfies condition $C$ on $S$ and $f:I\rightarrow \Bbb R$ a continuous function. Then, for every $A,B\in S$ with spectra of $A,V:=A+\eta(B,A)$ in $I$, the function $f$ is operator preinvex on $I$ with respect to $\eta$ for operators in $\eta-$path $P_{AV}$ if and only if the function $\varphi_{x,A,B}:[0,1]\rightarrow \Bbb R$ defined by
\begin{equation}\label{2.5}
\varphi_{x,A,B}(t):=\langle f(A+t\eta(B,A))x,x\rangle
\end{equation}
is convex on $[0,1]$ for every $x\in H$ with $\|x\|=1$.
\end{proposition}
\begin{proof}
Let the function $f$ be operator preinvex on $I$ with respect to $\eta$ for operators in $\eta-$path $P_{AV}$. Suppose that $A,B\in S$ with spectra $A,V$ in $I$, since $A+t\eta(B,A)=tV+(1-t)A$ therefore for all $t\in [0,1],~ Sp(A+t\eta(B,A))\subseteq I$. If $t_{1},t_{2}\in [0,1]$ since $\eta$ satisfies condition $C$ on $S$, we have
\begin{equation}\label {2.7}
\begin{aligned}
\varphi_{x,A,B}((1-\lambda) t_{1}&+\lambda t_{2})=\langle f(A+((1-\lambda) t_{1}+\lambda t_{2})\eta(B,A))x,x\rangle\\
&=\langle f\left(A+ t_{1}\eta(B,A)+\lambda\eta( A+ t_{2}\eta(B,A),A+ t_{1}\eta(B,A)\right)x,x\rangle\\
&\leq \lambda \langle f(A+ t_{2}\eta(B,A))x,x\rangle+(1-\lambda)\langle f(A+ t_{1}\eta(B,A))x,x\rangle\\
&=\lambda\varphi_{x,A,B}(t_{2})+(1-\lambda)\varphi_{x,A,B}(t_{1}),
\end{aligned}
\end{equation}
 for every $\lambda \in [0,1]$  and $x\in H$ with $\|x\|=1$. Therefore, $\varphi_{x,A,B}$ is convex on $[0,1]$.

Conversely, suppose that $x\in H$ with $\|x\|=1$ and $\varphi_{x,A,B}$ is convex on $[0,1]$ and $C_{1}:=A+t_{1}\eta(B,A)\in P_{AV}, C_{2}:=A+t_{2}\eta(B,A)\in P_{AV}$. Fix $\lambda \in [0,1]$.
By (\ref{2.4}) we have
\begin{equation}\label {2.6}
\begin{aligned}
\langle f(C_{1}+\lambda\eta(C_{2},C_{1}))x,x\rangle&=\langle f(A+((1-\lambda) t_{1}+\lambda t_{2})\eta(B,A))x,x\rangle\\
&=\varphi_{x,A,B}((1-\lambda) t_{1}+\lambda t_{2})\\
&\leq (1-\lambda)\varphi_{x,A,B}(t_{1})+\lambda\varphi_{x,A,B}(t_{2})\\
&=(1-\lambda) \langle f(C_{1})x,x\rangle+\lambda\langle f(C_{2})x,x\rangle.
\end{aligned}
\end{equation}
Hence, $f$ is operator preinvex with respect to $\eta$ for operators in $\eta-$path $P_{AV}$.
\end{proof}
\section{Hermite-Hadamard type inequalities}

In this section we generalize Theorem \ref{t1} and Theorem \ref{t3} for operator preinvex functions and establish an estimate for the right-hand side of the Hermite-Hadamard operator inequality for such functions.
Some Hermite-Hadamard type inequalities for the product of two operator preinvex functions is also given.

The following Theorem is a generalization of Theorem \ref{t3} for operator preinvex functions.
\begin{theorem}\label{t4}
Let $S\subseteq B(H)_{sa}$ be an invex set with respect to $\eta:S\times S\rightarrow B(H)_{sa}$ and $\eta$ satisfies condition $C$. If for every $A,B\in S$ with spectra of $A, V:=A+\eta(B,A)$ in the interval $I$, the function $f : I \rightarrow \mathbb{R}$ is operator preinvex with respect to $\eta$ for operators in  $\eta-$path $P_{AV}$ and $k, p$ are positive integers, then the following inequalities holds
\begin{multline}\label{2.8}
f\left(A+\frac{1}{2}\eta(B,A)\right)\leq\frac{1}{k^p}\sum_{i=0}^{k^p-1}f\left(A+\frac{2i+1}{2k^p}\eta(B,A)\right)\\
\leq\int_0^1 f(A+t\eta(B,A))dt\\
\leq\frac{1}{2k^p}\sum_{i=0}^{k^p-1}\left[f\left(A+\frac{i+1}{k^p}\eta(B,A)\right)
+f\left(A+\frac{i}{k^p}\eta(B,A)\right)\right]\\
\leq\frac{f(A)+f(B)}{2}.
\end{multline}
\end{theorem}

\begin{proof}
For $x\in H$ with $\|x\|=1$ and $t\in [0,1]$, we have
\begin{equation}\label{2.9}
\langle (A+t\eta(B,A))x,x\rangle =(1-t)\langle Ax,x\rangle+t\langle Vx,x\rangle \in I,
\end{equation}
since $\langle Ax,x\rangle\in Sp(A)\subseteq I$ and $\langle Vx,x\rangle\in Sp(V)\subseteq I$.

Continuity of $f$ and (\ref{2.9}) imply that the operator valued integral $\int_0^1 f(A+t\eta(B,A))dt$
exists. Since $\eta$ satisfied condition $C$, therefore by (\ref{2.2}) for every $t\in[0,1] $ we have
\begin{equation}
\begin{aligned}\label{2.10}
A+\frac{1}{2}\eta(B,A)&=A+t\eta(B,A)+\frac{1}{2}\eta(A+(1-t)\eta(B,A), A+t\eta(B,A)).
\end{aligned}
\end{equation}
Let $x\in H$ be a unit vector,
define the real-valued function $\varphi:[0,1]\rightarrow \mathbb{R}$ given by
$\varphi(t)=\langle f(A+t\eta(B,A))x,x\rangle$. Since $f$ is operator preinvex,
by the previous proposition \ref {p1}, $\varphi$ is a convex function on $[0,1]$.
Utilizing the classical Hermite-Hadamard inequality for real-valued convex
function $\varphi$ on the interval $[\frac{i}{k^p},~\frac{i+1}{k^p}]$, we get
\begin{equation}
\begin{aligned}\label{2.11}
\varphi\left(\frac{2i+1}{2k^p}\right)\leq k^p\int_{\frac{i}{k^p}}^{\frac{i+1}{k^p}} \varphi(t)dt\leq\frac{\varphi(\frac{i}{k^p})+\varphi(\frac{i+1}{k^p})}{2}
\end{aligned}
\end{equation}
Summation of the above inequalities over $i=0,...,k^p-1$ yields
\begin{equation}
\begin{aligned}\label{2.12}
\sum_{i=0}^{k^p-1}\varphi\left(\frac{2i+1}{2k^p}\right)\leq k^p\int_0^1\varphi(t)dt\leq\sum_{i=0}^{k^p-1}\frac{\varphi(\frac{i}{k^p})+\varphi(\frac{i+1}{k^p})}{2}.
\end{aligned}
\end{equation}
Hence
\begin{multline}\label{2.13}
\frac{1}{k^p}\sum_{i=0}^{k^p-1}f\left(A+\frac{2i+1}{2k^p}\eta(B,A)\right)\\
\leq\int_0^1 f(A+t\eta(B,A))dt\\
\leq\frac{1}{2k^p}\sum_{i=0}^{k^p-1}\left[f\left(A+\frac{i+1}{k^p}\eta(B,A)\right)
+f\left(A+\frac{i}{k^p}\eta(B,A)\right)\right].
\end{multline}
Inequality (\ref{2.10}) for $t=\frac{2i+1}{k^p}$ and operator preinvexity $f$ imply that
\begin{equation}
\begin{aligned}\label{2.14}
f&\left(A +\frac{1}{2}\eta(B,A)\right)\\
&\leq\frac{1}{2}f\left(A+\frac{2i+1}{k^p}\eta(B,A)\right)+\frac{1}{2}f\left(A+\left(1-\frac{2i+1}{k^p}\right)\eta(B,A)\right).
\end{aligned}
\end{equation}
Summation of the above inequalities over $i=0,...,k^p-1$ and the following equality \[\sum_{i=0}^{k^p-1}f\left(A+\frac{2i+1}{k^p}\eta(B,A)\right)=\sum_{i=0}^{k^p-1}f\left(A+\left(1-\frac{2i+1}{k^p}\right)\eta(B,A)\right)
\] yield
\begin{equation}
\begin{aligned}\label{2.15}
k^pf&\left(A +\frac{1}{2}\eta(B,A)\right)
\leq\sum_{i=0}^{k^p-1}f\left(A+\frac{2i+1}{k^p}\eta(B,A)\right).
\end{aligned}
\end{equation}
In the other hand, from preinvexity $f$ we have
\begin{multline}\label{2.16}
\frac{1}{2k^p}\sum_{i=0}^{k^p-1}\left[f\left(A+\frac{i+1}{k^p}\eta(B,A)\right)
+f\left(A+\frac{i}{k^p}\eta(B,A)\right)\right]\\
\leq \frac{1}{2k^p}\sum_{i=0}^{k^p-1}\left[\frac{i+1}{k^p}f(B)+\left(1-\frac{i+1}{k^p}\right)f(A)\right.
\left.+\frac{i}{k^p}f(B)+\left(1-\frac{i}{k^p}\right) f(A)\right]\\
\leq\frac{f(A)+f(B)}{2}.
\end{multline}
From inequalities (\ref{2.13}), (\ref{2.15}) and (\ref{2.16}) we obtain (\ref{2.8}).
\end{proof}

A simple consequence of the above theorem is that the integral is closer
to the left bound than to the right, namely we can state:
\begin{corollary}\label{c1}
With the assumptions in Theorem \ref{t4} we have the inequality
\begin{align*}
0\leq \int_0^1 f(A+t\eta(B,A))dt-f\left(A+\frac{1}{2}\eta(B,A)\right)\leq \frac{f(A)+f(B)}{2}-\int_0^1 f(A+t\eta(B,A))dt.
\end{align*}
\end{corollary}

\begin{example}
Let $S,~f,~\eta_1$ be as in Example \ref{e2}, then we have
\begin{equation*}
\left(A+\frac{1}{2}\eta_1(B,A)\right)^2\leq \int_0^1 (A+t\eta_1(B,A))^2dt\leq \frac{A^2+B^2}{2},
\end{equation*}
for every $A,B\in S$.
\end{example}

Let $S\subseteq B(H)_{sa}$ be an invex set with respect to $\eta:S\times S\rightarrow B(H)_{sa}$ and $f,g : I \rightarrow \mathbb{R}$ operator preinvex functions on the interval $I$ with respect to $\eta$ for operators in $\eta$-path $P_{AV}$. Then
for every $A,B\in S$ with spectra of $A, V:=A+\eta(B,A)$ in the interval $I$, we define real functions
$M(A,B)$ and $N(A,B)$  on $H$ by
\begin{align*}
M(A,B)(x)&=\langle f(A)x,x\rangle\langle g(A)x,x\rangle+\langle f(B)x,x\rangle\langle g(B)x,x\rangle\quad &&(x\in H),\\
N(A,B)(x)&=\langle f(A)x,x\rangle\langle g(B)x,x\rangle+\langle f(B)x,x\rangle\langle g(A)x,x\rangle\quad &&(x\in H).
\end{align*}

The following Theorem is a generalization of Theorem \ref{t1} for operator preinvex functions.
\begin{theorem}\label{t5}
Let $f,g : I \rightarrow \mathbb{R^+}$ be operator preinvex functions on the interval $I$ with respect to $\eta$
and $\eta$ satisfies condition $C$. Then
for any selfadjoint operators $A$ and $B$ on a Hilbert space $H$ with spectra $A,~ V$ in $I$, the inequality

\begin{multline}\label{2.17}
\int_{0}^{1}\langle f(A+t\eta(B,A))x,x\rangle\langle g(A+t\eta(B,A))x,x\rangle dt\\
\leq\frac{1}{3}M(A,B)(x)+\frac{1}{6}N(A,B)(x),
\end{multline}

holds for any $x\in H$ with $\|x\| = 1$.
\end{theorem}

\begin{proof}
Continuity of $f,g$ and (\ref{2.9}) imply that the following operator valued integrals exist
\[
\int_0^1 f(B+t\eta(B,A))dt, \int_0^1 g(A+t\eta(B,A))dt, \int_0^1 (fg)(A+t\eta(B,A))dt.
\]

Since $f$ and $g$ are operator preinvex, therefore for $t$ in $[0,1]$ and $x\in H$ we have
\begin{equation}\label{2.18}
\langle f(A+t\eta(B,A))x,x\rangle\leq\langle (tf(B)+(1-t)f(A))x,x\rangle,
\end{equation}
\begin{equation}\label{2.19}
\langle g(A+t\eta(B,A))x,x\rangle\leq\langle(tg(B)+(1-t)g(A))x,x\rangle.
\end{equation}

From (\ref{2.18}) and (\ref{2.19}) we obtain

\begin{multline}\label{2.20}
\langle f(A+t\eta(B,A))x,x\rangle\langle g(A+t\eta(B,A))x,x\rangle\\
\leq (1-t)^2\langle f(A)x,x\rangle\langle g(A)x,x\rangle+t^2\langle f(B)x,x\rangle\langle g(B)x,x\rangle\\
+t(1-t)\left[\langle f(A)x,x\rangle\langle g(B)x,x\rangle+\langle f(B)x,x\rangle \langle g(A)x,x\rangle\right].
\end{multline}

Integrating both sides of (\ref{2.20}) over $[0,1]$ we get the required inequality (\ref{2.17}).
\end{proof}

\begin{theorem}\label{t6}
Let $f,g : I \rightarrow \mathbb{R}$ be operator preinvex functions on the interval $I$ with respect to $\eta$.
 If $\eta$ satisfies condition $C$, then
for any selfadjoint operators $A$ and $B$ on a Hilbert space $H$ with spectra $A,~ V$ in $I$, the inequality

\begin{multline}\label{2.21}
\left\langle f\left(A+\frac{1}{2}\eta(B,A)\right)x,x\right\rangle\left\langle g\left(A+\frac{1}{2}\eta(B,A)\right)x,x\right\rangle\\
\leq\frac{1}{2}\int_{0}^{1}\left\langle f(A+t\eta(B,A))x,x\rangle\langle g(A+t\eta(B,A))x,x\right\rangle dt \\
+\frac{1}{12}M(A,B)(x)+\frac{1}{6}N(A,B)(x),
\end{multline}

holds for any $x\in H$ with $\|x\| = 1$.
\end{theorem}

\begin{proof}Put $D=A+t\eta(B,A)$ and $E=A+(1-t)\eta(B,A)$, by (\ref{2.10}) we have $A+\frac{1}{2}\eta(B,A)=D+\frac{1}{2}\eta(E,D)$.
Since $f$ and $g$ are operator preinvex, therefore for any $t\in I$ and any $x\in H$ with $\|x\| = 1$ we observe that
\begin{multline}\label{2.22}
\left\langle f\left(A+\frac{1}{2}\eta(B,A)\right)x,x\right\rangle \left\langle g\left(A+\frac{1}{2}\eta(B,A)\right)x,x\right\rangle\\
=\left\langle f\left(D+\frac{1}{2}\eta(E,D)\right)x,x\right\rangle\left\langle g\left(D+\frac{1}{2}\eta(E,D)\right)x,x\right\rangle\\
\leq\left\langle \left(\frac{f(D)+f(E)}{2}\right)x,x\right\rangle\left\langle \left(\frac{g(D)+g(E)}{2}\right)x,x\right\rangle\\
\leq\frac{1}{4}[\langle f(D)x,x\rangle+\langle f(E)x,x\rangle][\langle g(D)x,x\rangle+\langle g(E)x,x\rangle]\\
\leq\frac{1}{4}[\langle f(D)x,x\rangle \langle g(D)x,x\rangle
+\langle f(E)x,x\rangle \langle g(E)x,x\rangle]\\
+\frac{1}{4}[t\langle f(A)x,x\rangle+(1-t)\langle f(B)x,x\rangle][(1-t)\langle g(A)x,x\rangle+t\langle g(B)x,x\rangle]\\
+[(1-t)\langle f(A)x,x\rangle+t\langle f(B)x,x\rangle][t\langle g(A)x,x\rangle+(1-t)\langle g(B)x,x\rangle]\\
=\frac{1}{4}[\langle f(D)x,x\rangle \langle g(D)x,x\rangle+\langle f(E)x,x\rangle\langle g(E)x,x\rangle]\\
+\frac{1}{4}2t(1-t)[\langle f(A)x,x\rangle\langle g(A)x,x\rangle+\langle f(B)x,x\rangle \langle g(B)x,x\rangle]\\
+(t^2+(1-t)^2)[\langle f(A)x,x\rangle \langle g(B)x,x\rangle+\langle f(B)x,x\rangle \langle g(A)x,x\rangle].
\end{multline}
We integrate both sides of (\ref{2.22}) over [0,1] and obtain
\begin{multline*}
\left\langle f\left(A+\frac{1}{2}\eta(B,A)\right)x,x\right\rangle \left\langle g\left(A+\frac{1}{2}\eta(B,A)\right)x,x\right\rangle\\
\leq\frac{1}{4}\int_{0}^{1}[\langle f(A+t\eta(B,A))x,x\rangle\langle g(A+t\eta(B,A))x,x\rangle\\
+\langle f(A+(1-t)\eta(B,A))x,x\rangle\langle g(A+(1-t)\eta(B,A))x,x\rangle]dt\\
+\frac{1}{12}M(A,B)(x)+\frac{1}{6}N(A,B)(x).
\end{multline*}

This implies the required inequality (\ref{2.21}).
\end{proof}

The following Theorem is a generalization of Theorem 3.1 in \cite{bara}.
\begin{theorem}\label{t6}
Let the function $f:I\rightarrow \mathbb{R^+}$ is continuous, $S\subseteq B(H)_{sa}$ be an open invex set with respect to $\eta:S\times S\rightarrow B(H)_{sa}$ and $\eta$ satisfies condition $C$. If
for every $A,B\in S$ and $V=A+\eta(B,A)$ the function $f$  is operator preinvex with respect to $\eta$ on $\eta-$path $P_{AV}$ with spectra of $A$ and $V$ in $I$. Then, for every $a,b\in (0,1)$ with $a<b$ and every $x\in H$ with $\|x\|=1$ the following inequality holds,
\begin{multline}\label{2.23}
\left|\frac{1}{2}\left\langle\int_0^a f(A+s\eta(B,A))ds~x,x\right\rangle +\frac{1}{2}\left\langle\int_0^b f(A+s\eta(B,A))ds~x,x\right\rangle\right. \\
\left. -\frac{1}{b-a}\int_a^b\left\langle\int_0^t f(A+s\eta(B,A))ds~x,x\right\rangle dt\right|\\
\leq \frac{b-a}{8} \{\langle f(A+a\eta(B,A))x,x\rangle+ \langle f(A+b\eta(B,A))x,x\rangle\}.
\end{multline}
Moreover we have
\begin{multline}\label{2.24}
\left\|\frac{1}{2}\int_0^a f(A+s\eta(B,A))ds +\frac{1}{2}\int_0^b f(A+s\eta(B,A))ds\right. \\
\left. -\frac{1}{b-a}\int_a^b\int_0^t f(A+s\eta(B,A))ds dt \right\|\\
\leq \frac{b-a}{8} \| f(A+a\eta(B,A))+  f(A+b\eta(B,A))\|\\
\leq \frac{b-a}{8}[~ \| f(A+a\eta(B,A))\|+\|f(A+b\eta(B,A))\|~].
\end{multline}
\end{theorem}

\begin{proof}
Let $A,B\in S$ and $a,b\in (0,1)$ with $a<b$. For $x\in H$ with $\|x\|=1$ we define the function $\varphi:[0,1]\rightarrow{\Bbb R}^{+}$ by

\[
\varphi(t):=\left\langle \int_0^tf(A+s\eta(B,A))ds~x,x\right\rangle.
\]
Utilizing the continuity of the function $f$, the continuity property of
the inner product and the properties of the integral of operator-valued functions
we have
\[
\left\langle \int_0^t f(A+s\eta(B,A))ds~x,x\right\rangle=\int_0^t\left\langle f(A+s\eta(B,A))~x,x\right\rangle ds.
\]
Since $f(A+s\eta(B,A))\geq 0$, therefore $\varphi(t)\geq 0$ for all $t\in I$.
Obviously for every  $t\in(0,1)$ we have
\[
\varphi^{\prime}(t)=\langle f(A+t\eta(B,A))x,x\rangle\geq0,
\]
hence, $|\varphi^{\prime}(t)|=\varphi^{\prime}(t)$. Since $f$ is operator preinvex with respect to $\eta$ on $\eta-$path $P_{AV}$, by Proposition \ref{p1} the function $\varphi^{\prime}$ is convex. Applying Theorem 2.2 in \cite{dra3} to the function $\varphi$ implies that
\[
\left| \frac{\varphi(a)+\varphi(b)}{2}-\frac{1}{b-a}\int_a^b\varphi(s)ds\right|\leq\frac{(b-a)\left(\varphi^{\prime}(a)+\varphi^{\prime}(b)\right)}{8},
\]
and we deduce that (\ref{2.23}) holds.
Taking supremum  over both side of inequality (\ref {2.23}) for all $x$ with $\|x\|=1$,  we deduce that the inequality (\ref{2.24}) holds.
\end{proof}

\bibliographystyle{amsplain}

\end{document}